\documentclass[a4paper,12pt]{amsart}
\usepackage{amscd}
\usepackage{fullpage}

\newtheorem{theorem}{Theorem}[section]
\newtheorem{lemma}[theorem]{Lemma}

\newtheorem*{theorema}{Theorem A}
\newtheorem*{theoremb}{Theorem B}
\newtheorem*{theoremc}{Theorem C}

\newcommand{\Z}{\mathbb{Z}}
\newcommand{\N}{\mathbb{N}}

\DeclareMathOperator{\Ker}{Ker}
\DeclareMathOperator{\Coker}{Coker}
\DeclareMathOperator{\Ima}{Im}

\DeclareMathOperator{\Hom}{Hom}
\DeclareMathOperator{\End}{End}
\DeclareMathOperator{\Ext}{Ext}
\DeclareMathOperator{\ext}{ext}
\DeclareMathOperator{\Tor}{Tor}

\DeclareMathOperator{\dimv}{\underline{\dim}}

\DeclareMathOperator{\Spec}{Spec}
\DeclareMathOperator{\Rep}{Rep}
\DeclareMathOperator{\GL}{GL}
\DeclareMathOperator{\pd}{proj.dim}

\begin{document}
\title{Rigid integral representations of quivers over arbitrary commutative rings}

\author{William Crawley-Boevey}
\address{Fakult\"at f\"ur Mathematik, Universit\"at Bielefeld, 33501 Bielefeld, Germany}
\email{wcrawley@math.uni-bielefeld.de}





\thanks{Partially supported by the Deutsche Forschungsgemeinschaft (DFG, German Research Foundation) -- SFB-TRR 358/1 2023 -- 491392403}

\begin{abstract}
In earlier work, the author classified rigid representations of a quiver by finitely generated free modules over a principal ideal ring. Here we extend the results to representations of a quiver by finitely generated projective modules over an arbitrary commutative ring.
\end{abstract}

\maketitle

\section{Introduction}
In \cite{CBrirq} we studied rigid representations of a quiver over a principal ideal domain. Here we generalize the results to arbitrary commutative base rings $R$ (always unital and non-zero). Let $Q=(Q_0,Q_1,h,t)$ be a finite quiver. We are interested in representations of $Q$ in the category of finitely generated projective $R$-modules, so given by a finitely generated projective $R$-module $X_i$ for each vertex $i\in Q_0$ and an $R$-module homomorphism $X_{t(a)}\to X_{h(a)}$ for each arrow $a\in Q_1$. Letting $RQ$ be the path algebra of $Q$ over $R$, it is equivalent to consider \emph{$RQ$-lattices}, that is (left) $RQ$-modules $X$ which are finitely generated projective as an $R$-module; the $R$-module associated to vertex $i$ is then $X_i = e_i X$, where $e_i$ is the trivial path at $i$. An $RQ$-module $X$ is said to be \emph{rigid} if $\Ext^1_{RQ}(X,X)=0$. Our first result generalizes \cite[Lemma~2]{CBrirq}.

\begin{theorema}
If $X$ and $Y$ are rigid $RQ$-lattices, then $\Hom_{RQ}(X,Y)$ and $\Ext^1_{RQ}(X,Y)$ are finitely generated projective $R$-modules.
\end{theorema}

Recall that a finitely generated projective $R$-module $P$ has \emph{constant rank $n$} if $P_{\mathfrak{p}}$ is a free $R_{\mathfrak{p}}$-module of rank $n$ for each prime ideal $\mathfrak{p}$ of $R$, or equivalently for each maximal ideal. If $\Spec R$ is connected, or equivalently $R$ has no idempotents other than~0 and~1, then every finitely generated projective $R$-module has constant rank, see for example \cite[Exercise 20.12]{E}. The projective modules of constant rank 1 are the invertible modules, whose isomorphism classes are indexed by the Picard group. We say that an $RQ$-lattice $X$ has \emph{rank vector $\alpha\in\N^{Q_0}$} if $e_iX$ has constant rank $\alpha_i$ for each vertex $i$, and that $X$ has \emph{pointwise constant rank} if it has rank vector $\alpha$ for some $\alpha$; this holds if $X$ is \emph{pointwise free}, that is, $e_i X$ is a free $R$-module for all $i$. 

An $RQ$-module $X$ is said to be \emph{exceptional} if it is rigid and the natural map $R\to \End_{RQ}(X)$ is an isomorphism. If $K$ is an algebraically closed field, then by definition the possible dimension vectors of exceptional $KQ$-modules are the \emph{real Schur roots} for $Q$. By results of \cite{CBesrq}, \cite{CBsgrq} or \cite{CK} these do not depend on the choice of the algebraically closed field $K$. Using \cite{CBesrq}, Hubery and Krause characterize the real Schur roots as the positive real roots for which the corresponding reflection is a non-crossing partition \cite[Corollary 4.8]{HubKr}.

If $P$ is an $R$-module and $X$ is an $RQ$-module, then $X\otimes_R P$ is naturally an $RQ$-module. If $P$ is finitely generated projective and $X$ is a lattice, then so is $X\otimes_R P$.

\begin{theoremb}
There is an exceptional $RQ$-lattice of rank vector $\alpha\in\N^{Q_0}$ if and only if $\alpha$ is a real Schur root. In this case there is a unique exceptional pointwise free $RQ$-lattice $X$ of rank vector $\alpha$, and any rigid $RQ$-lattice of rank vector $\alpha$ is exceptional and isomorphic to $X\otimes_R P$ with $P$ a finitely generated projective $R$-module of constant rank~1. Moreover $P$ is uniquely determined up to isomorphism.
\end{theoremb}

The existence follows easily from results in \cite{CBrirq}. Concerning uniqueness, the case when $R$ is a principal ideal domain is in \cite{CBrirq}, and the case when $R$ is a truncated polynomial ring $K[\epsilon]/(\epsilon^n)$ is a special case of \cite[Theorem 1.2]{GLS}.

\begin{theoremc}
An $RQ$-lattice $X$ of pointwise constant rank is rigid if and only if it has a decomposition
\[
X \cong (X_1\otimes_R P_1) \oplus \dots \oplus (X_r \otimes_R P_r)
\]
where the $X_i$ are pairwise non-isomorphic exceptional pointwise free $RQ$-lattices satisfying $\Ext_{RQ}^1(X_i,X_j)=0$ for all $i,j$ and the $P_i$ are non-zero finitely generated projective $R$-modules of constant rank. Moreover this decomposition is unique up to isomorphism and reordering. 
\end{theoremc}

Note that in general one doesn't have uniqueness for rigid pointwise free lattices of a given rank vector. For example let $Q$ be the quiver $1\to 2$ and $R$ a reduced ring with a stably free projective module $P$ which is not free, say $P \oplus R^n \cong R^m$. Then $(RQe_2 \otimes_R P) \oplus (RQe_1 \otimes_R R^n)$ and $(RQe_2 \otimes_R R^{m-n}) \oplus (RQe_1 \otimes_R R^n)$ are both rigid and pointwise free of rank vector $(n,m)$, but not isomorphic.

\section{Commutative rings}
If $R\to S$ is a homomorphism of commutative rings, and $X$ is an $R$-module, then there is an induced $S$-module $X^S = S\otimes_R X$. Our theorems are deduced from results about representations of quivers over fields, often algebraically closed, and so we are particularly interested in homomorphisms $R\to K$ with $K$ an algebraically closed field. Clearly if $P$ is a finitely generated projective $R$-module, then $P$ has constant rank $n$ if and only if $\dim_K P^K = n$ for all such homomorphisms, and in this case $P^S$ also has constant rank $n$.

\begin{lemma}
\label{l:commring}
Let $R$ be a commutative ring.
\begin{itemize}
\item[(i)]If $X$ is a finitely generated $R$-module, then $X = 0$ if and only if $X^K=0$ for all homomorphisms $R\to K$ with $K$ an algebraically closed field.

\item[(ii)]If $\theta:X\to Y$ is a homomorphism of $R$-modules and $Y$ is finitely generated, then $\theta$ is onto if and only if $\theta^K:X^K\to Y^K$ is onto for all $R\to K$ with $K$ an algebraically closed field.
\end{itemize}
\end{lemma}

\begin{proof}
(i) is essentially Nakayama's Lemma, and (ii) follows.
\end{proof}

Recall that a commutative ring $R$ is \emph{reduced} if it has no non-zero nilpotent elements.
If $\mathfrak{p}$ is a prime ideal in $R$, we write $\overline{k(\mathfrak{p})}$ for an algebraic closure of the quotient field of $R/\mathfrak{p}$.

\begin{lemma}
\label{l:redring}
Let $R$ be a reduced commutative ring.
\begin{itemize}
\item[(i)]If $X$ is a finitely generated $R$-module, then $X$ is projective if and only if the function 
$\Spec R\to \Z$, $\mathfrak{p}\mapsto \dim_{\overline{k(\mathfrak{p})}} X^{\overline{k(\mathfrak{p})}}$ is locally constant.

\item[(ii)]If $X$ is a finitely generated $R$-module, then $X$ is projective of constant rank $n$ if and only if $\dim_K X^K = n$ for all homomorphisms $R\to K$ with $K$ an algebraically closed field.

\item[(iii)]If $R$ is reduced and $\theta:X\to Y$ is a homomorphism with $X,Y$ finitely generated projective, then $\theta$ is a split monomorphism if and only if $\theta^K$ is a monomorphism for all $R\to K$ with $K$ an algebraically closed field.
\end{itemize}
\end{lemma}

\begin{proof}
(i) and (ii) follow from \cite[Exercise 20.13]{E}. (iii) If $\theta^K$ is a monomorphism for all $R\to K$, then the mapping $\mathfrak{p}\mapsto \Coker \theta^{\overline{k(\mathfrak{p})}}$ has locally constant dimension, so $\Coker \theta$ is projective. Then $\Ima \theta$ is a summand of $Y$, so projective. But then the sequence $0\to \Ker \theta\to X\to \Ima \theta\to 0$ splits, so is exact after inducing to $K$. Thus $\Ker \theta^K=0$. Also $\Ker \theta$ is summand of $X$, so finitely generated. Thus $\Ker \theta$ is zero.
\end{proof}

\begin{lemma}
\label{l:idempscheme}
If $n\ge 0$ and $R$ is reduced, then so is the ring 
\[
R' = R[x_{ij}:1\le i,j\le n]/\left( \sum_{\ell=1}^n x_{i\ell}x_{\ell j} - x_{ij} : 1\le i,j\le n\right)
\]
\end{lemma}

\begin{proof}
$R'$ is the coordinate ring of the scheme of $n\times n$ idempotent matrices over $R$. That is, for any commutative $R$-algebra $C$, there is a bijection between set $\Hom(R',C)$ of $R$-algebra homomorphisms $R'\to C$ and the set of idempotent matrices in $M_n(C)$.

Now if $I$ is a square-zero ideal in $C$, then the map $\Hom(R',C)\to \Hom(R',C/I)$ is surjective, since any idempotent in $M_n(C/I) \cong M_n(C)/M_n(I)$ lifts modulo the square-zero ideal $M_n(I)$ to an idempotent in $M_n(C)$. 

It follows that the morphism of schemes $\Spec R'\to \Spec R$ is smooth, see \cite[Corollaire 4.6]{DG}. Since $R$ is reduced, it follows that $R'$ is reduced \cite[Corollaire 6.3]{DG}.
\end{proof}

We write $J(R)$ for the Jacobson radical of $R$.

\begin{lemma}
\label{l:liftprojs}
Let $R$ be a commutative ring and let $(P_\lambda)_{\lambda\in\Lambda}$ be a family of finitely generated projective $R$-modules. Then there is a commutative ring $R'$, a homomorphism $\theta:R'\to R$ and finitely generated projective $R'$-modules $(P_\lambda')_{\lambda\in\Lambda}$ such that
\begin{itemize}
\item[(i)]
$R'$ is reduced and $\theta$ is surjective with kernel contained in $J(R')$.
\item[(ii)]
$P_\lambda \cong (P_\lambda')^R$ for all $\lambda\in\Lambda$.
\end{itemize}
\end{lemma}

\begin{proof}
Let $A = \Z[x_r : r\in R]$ be the polynomial ring over $\Z$ with indeterminates indexed by the elements of $R$. Then $A$ is a domain and there is a surjective homomorphism $\theta_A:A\to R$ sending each $x_r$ to $r$.

For each $\lambda\in\Lambda$, write $P_\lambda$ as a direct summand of a free $R$-module of rank $n_\lambda$ and let $e^\lambda \in M_{n_\lambda}(R)$ be the idempotent matrix corresponding to the projection onto $P_\lambda$. Let
\[
B = A[x_{ij}^\lambda : \lambda\in\Lambda, 1\le i,j\le n_\lambda] / \left( \sum_{\ell =1}^{n_\lambda} x^\lambda_{i\ell }x^\lambda_{\ell j} - x^\lambda_{ij} : \lambda\in\Lambda, 1\le i,j\le n_\lambda \right).
\]
By iterated use of Lemma~\ref{l:idempscheme} the ring obtained by adjoining the indeterminates $x^\lambda_{ij}$ and corresponding relations for a finite subset of $\Lambda$ is reduced, and it follows that $B$ is also reduced. There is a homomorphism $\theta_B:B\to R$ extending $\theta_A$ and sending $x^\lambda_{ij}$ to the $(i,j)$ entry in the matrix $e^\lambda$. Let
\[
S = \{ b\in B : \text{$\theta_B(b)$ is invertible} \},
\]
a multiplicative subset of $B$. We set $R'$ to be the localization $B_S$, which is reduced. The homomorphism $\theta_B$ extends to a surjective homomorphism $\theta:R'\to R$.

Let $x\in \Ker \theta$. If $y\in R'$, then $xy\in \Ker\theta$. Write $xy = bs^{-1}$ with $b\in B$ and $s\in S$. Then $\theta(b)=0$, so $s+b\in S$. Thus $1+xy = (s+b)s^{-1}$ is invertible in $R'$. Since this is true for all $y$, it follows that $x$ is in $J(R')$.

For each $\lambda\in \Lambda$, the matrix with entries $x^\lambda_{ij}$ is an idempotent matrix in $M_{n_\lambda}(R')$ which is sent to $e^\lambda$ by $\theta$, and the image of the corresponding endomorphism of $(R')^{n_\lambda}$, is a finitely generated projective $R'$-module $P'_\lambda$ with $(P_\lambda')^R \cong P_\lambda$.
\end{proof}

Of course by taking a suitable family $(P_\lambda)_{\lambda\in\Lambda}$ in the lemma, we could ensure that all finitely generated projective $R$-modules lift to $R'$.

\begin{lemma}
\label{l:nakprops}
Suppose that $P$ and $P'$ are finitely generated projective $R$-modules and that $\theta:R\to S$ is a surjective homomorphism with $\Ker \theta\subseteq J(R)$.
\begin{itemize}
\item[(i)]
If $X$ is a finitely generated $R$-module and $X^S=0$, then $X=0$.
\item[(ii)]
If $P^S$ has constant rank $n$, then so does $P$.
\item[(iii)]
Any $S$-module map $P^S \to (P')^S$ lifts to an $R$-module map $P\to P'$.
\item[(iv)]
If $P^S \cong (P')^S$, then $P\cong P'$.
\end{itemize}
\end{lemma}

\begin{proof}
(i) is Nakayama's Lemma. For (ii) it suffices to check that the free $R_{\mathfrak{m}}$-module $P_{\mathfrak{m}}$ has rank $n$ for all maximal ideals $\mathfrak{m}$ in $R$, but this holds since the homomorphism $R\to R/\mathfrak{m}$ factors through $S$. Part (iii) is clear, since $P$ is projective and the mapping 
\[
P' \to (P')^S \cong P'/(\Ker\theta)P'
\]
is surjective. For (iv), an isomorphism $P^S \to (P')^S$ lifts to a homomorphism $f:P \to P'$. Now $\Coker(f)^S=0$, so $f$ is surjective by (i). Then the exact sequence $0\to \Ker f\to P\to P'\to 0$ splits, so remains exact on inducing to $S$. Thus $(\Ker f)^S=0$, so $f$ is a monomorphism by (i).
\end{proof}

\section{Lattices}
If $X$ is an $RQ$-module and $R\to S$ is a homomorphism of commutative rings, then $X^S$ is naturally an $SQ$-module, and if $X$ is an $RQ$-lattice, then $X^S$ is an $SQ$-lattice. Lemma~1 of \cite{CBrirq} generalizes immediately.

\begin{lemma}
\label{l:extprop}
Suppose $X$ is an $RQ$-lattice and $Y$ is an $RQ$-module. The following hold.
\begin{itemize}
\item[(i)]
There is a resolution $0\to P_1\to P_0\to X\to 0$ of $X$ by finitely generated projective $RQ$-modules. In particular, $\pd X \le 1$
\item[(ii)]
If $Y$ is finitely generated as an $R$-module, then so is $\Ext^1_{RQ}(X,Y)$.
\item[(iii)]
$\Ext^1_{RQ}(X,Y)^S \cong \Ext^1_{SQ}(X^S,Y^S)$ for any homomorphism $R\to S$ and any $RQ$-module $Y$.
In particular, if $X$ is a rigid $RQ$-lattice, then $X^S$ is a rigid $SQ$-lattice.
\item[(iv)]
If $X^K$ is rigid for all homomorphisms $R\to K$ with $K$ an algebraically closed field, then $X$ is rigid.
\item[(v)]
If $P,P'$ are finitely generated projective $R$-modules, then
\begin{align*}
\Hom_{RQ}(X\otimes_R P,Y\otimes_R P') &\cong \Hom_{RQ}(X,Y)\otimes_R \Hom_R(P,P'),\text{ and}
\\
\Ext^1_{RQ}(X\otimes_R P,Y\otimes_R P') &\cong \Ext^1_{RQ}(X,Y)\otimes_R \Hom_R(P,P').
\end{align*}
\end{itemize}
\end{lemma}

\begin{proof}
(i) Letting $S$ be the $R$-subalgebra of $RQ$ with basis the trivial paths $e_i$ and $B$ be the free $R$-submodule of $RQ$ with basis the arrows, the path algebra $RQ$ is isomorphic to the tensor algebra of $B$ over $S$, so there is a standard resolution
\[
0 \to RQ \otimes_S B \otimes_S RQ \to RQ \otimes_S RQ \to RQ \to 0,
\]
or equivalently
\[
0 \to \bigoplus_{a\in Q_1} RQ e_{h(a)} \otimes_R e_{t(a)} RQ \to \bigoplus_{i\in Q_0} RQ e_i \otimes_R e_i R \to RQ \to 0.
\]
Tensoring with the $RQ$-module $X$ gives
\[
0 = \Tor_1^{RQ}(RQ,X) \to \bigoplus_{a\in Q_1} RQ e_{h(a)} \otimes_R e_{t(a)} X \to \bigoplus_{i\in Q_0} RQ e_i \otimes_R e_i X \to X \to 0.
\]
and this is a resolution of $X$ by finitely generated projective $RQ$-modules. 
Now (ii) follows directly from (i).
For (iii), since $X$ is projective over $R$, the induced sequence
\[
0\to P_1^S\to P_0^S \to X^S\to 0
\]
is exact, so a resolution of $X^S$ by projective $SQ$-modules. 
Now if $P$ is a finitely generated projective $RQ$-module, then
\[
\Hom_{RQ}(P,Y)^S \cong \Hom_{SQ}(P^S,Y^S).
\]
Thus we get a commutative diagram with exact rows
\[
\begin{CD}
\Hom_{RQ}(P_0,Y)^S @>>> \Hom_{RQ}(P_1,Y)^S @>>> \Ext^1_{RQ}(X,Y)^S @>>> 0 \\
@VVV @VVV \\
\Hom_{SQ}(P_0^S,Y^S) @>>> \Hom_{SQ}(P_1^S,Y^S) @>>> \Ext^1_{SQ}(X^S,Y^S) @>>> 0 
\end{CD}
\]
in which the vertical maps are isomorphisms, and hence an isomorphism
\[
\Ext^1_{RQ}(X,Y)^S \cong \Ext^1_{SQ}(X^S,Y^S).
\]
Now (iv) follows from (ii), (iii) and Lemma~\ref{l:commring}(i). For part (v), note that $P_i\otimes_R P$ is a finitely generated projective $RQ$-module, so the tensor product of the resolution in (i) with $P$ is a projective resolution of $X\otimes_R P$. The assertion follows.
\end{proof}

\begin{lemma}
\label{l:liftlatts}
Given $RQ$-lattices $X_1,\dots,X_n$, there is commutative ring $R'$, a homomorphism $\theta:R'\to R$ 
and $R'Q$-lattices $X_1',\dots,X_n'$ with the following properties.
\begin{itemize}
\item[(i)]
$R'$ is reduced and $\theta$ is surjective with kernel contained in $J(R')$.
\item[(ii)]
$X_i \cong (X_i')^R$ for all $i$.
\item[(iii)]
$X_i'$ has pointwise constant rank if and only if $X_i$ has pointwise constant rank, and if so, they have the same rank vector.
\item[(iv)]
$\Ext^1_{RQ}(X_i,X_j)=0$ if and only if $\Ext^1_{R'Q}(X'_i,X'_j)=0$. In particular $X_i$ is rigid if and only if $X_i'$ is rigid.
\end{itemize}
\end{lemma}

\begin{proof}
We apply Lemma \ref{l:liftprojs} to the ring $R$ and the finitely generated projective modules $X_{ij} = e_j X_i$ for $j\in Q_0$ and $i=1,\dots,n$, to obtain a ring $R'$ and homomorphism $\theta$ satisfying (i),
and finitely generated projective $R'$-modules $X'_{ij}$.

The lattice $X_i$ corresponds to a representation of $Q$ by means of the finitely generated projective $R$-modules $e_j X_i$ and an $R$-module homomorphism $e_{t(a)} X_i \to e_{h(a)} X_i$ for each arrow $a\in Q_1$. Now the $e_j X_i$ lift to modules $X'_{ij}$, and by Lemma~\ref{l:nakprops}(iii), the homomorphisms lift to homomorphisms $X'_{i,t(a)}\to X'_{i,h(a)}$. These give $R'Q$-lattices $X'_i$ satisfying~(ii).

Now (iii) follows from Lemma~\ref{l:nakprops}(ii) and part (iv) from Lemmas~\ref{l:nakprops}(i) and~\ref{l:extprop}(iii). (The idea of lifting rigid lattices modulo an ideal contained in the Jacobson radical, we learnt from \cite{GLS}.)
\end{proof}

Over an algebraically closed field $K$, there is a bijection between the isomorphism classes of representations $X$ of $Q$ of dimension vector $\alpha$ and orbits of a group $\GL(\alpha)$ acting on an affine space $\Rep(Q,\alpha)$. Moreover $X$ is rigid if and only if the corresponding orbit $O_X$ is open. Since the affine space it irreducible, it follows that there is at most one rigid representation of dimension $\alpha$, up to isomorphism. 

The general dimension of $\Hom(X,Y)$ and $\Ext^1(X,Y)$ for $X$ of dimension $\alpha$ and $Y$ of dimension $\beta$ is denoted $\hom(\alpha,\beta)$ and $\ext(\alpha,\beta)$. In particular these are the dimensions for $X$ and $Y$ rigid. Based on work of Schofield \cite{S}, it is proved in \cite{CBsgrq} that $\hom(\alpha,\beta)$ and $\ext(\alpha,\beta)$ do not depend on the algebraically closed field~$K$.

\section{Proofs of the theorems}
Here is a more detailed version of Theorem A.

\begin{theorem}
\label{t:homprop}
Let $X$ and $Y$ be rigid $RQ$-lattices.
\begin{itemize}
\item[(i)]
$\Ext^1_{RQ}(X,Y)$ and $\Hom_{RQ}(X,Y)$ are finitely generated projective $R$-modules. 
\item[(ii)]
$\Hom_{RQ}(X,Y)^S \cong \Hom_{SQ}(X^S,Y^S)$ for all homomorphisms $R\to S$. 
\item[(iii)]
If $X$ and $Y$ have rank vectors $\alpha$ and $\beta$, then $\Ext^1_{RQ}(X,Y)$ and $\Hom_{RQ}(X,Y)$ have constant ranks $\ext(\alpha,\beta)$ and $\hom(\alpha,\beta)$.
\end{itemize}
\end{theorem}

\begin{proof}
We prove the assertions first in case $R$ is reduced. Since $X$ is a lattice, the mapping $\Spec R\to \N^{Q_0}$, $\mathfrak{p}\mapsto \dimv X^{\overline{k(\mathfrak{p})}}$ is locally constant by Lemma~\ref{l:redring}(i). Similarly for $Y$. Thus by the remarks after Lemma~\ref{l:liftlatts}, the mapping 
\[
\Spec R\to \N, \quad 
\mathfrak{p}\mapsto \dim_{\overline{k(\mathfrak{p})}} \Ext^1_{\overline{k(\mathfrak{p})} Q}(X^{\overline{k(\mathfrak{p})}},Y^{\overline{k(\mathfrak{p})}})
= \dim_{\overline{k(\mathfrak{p})}} \Ext^1_{RQ}(X,Y)^{\overline{k(\mathfrak{p})}}
\]
is also locally constant, where the equality is given by Lemma~\ref{l:extprop}(iii). Now $\Ext^1_{RQ}(X,Y)$ is finitely generated by Lemma~\ref{l:extprop}(ii) and projective by Lemma~\ref{l:redring}(i). Now the projective resolution of Lemma~\ref{l:extprop}(i) gives an exact sequence
\[
0\to \Hom_{RQ}(X,Y) \to \Hom_{RQ}(P_0,Y) \to \Hom_{RQ}(P_1,Y) \to \Ext^1_{RQ}(X,Y) \to 0.
\]
The last three terms are finitely generated projective $R$-modules, so this sequence splits, and the first term is also finitely generated projective over $R$, giving (i). Moreover the sequence remains exact under induction using a homomorphism $R\to S$, from which (ii) follows. In (iii) the condition on $\Ext^1$ holds by the same argument as (i), and the condition on $\Hom$ follows from (ii) and Lemma~\ref{l:redring}(ii).

Now suppose that $R$ is not reduced. By Lemma~\ref{l:liftlatts}, we can find a surjective homomorphism $R'\to R$, with with kernel contained in the radical of $R'$ and with $R'$ reduced, such that $X$ and $Y$ lift to $R'Q$-lattices $X'$ and $Y'$. Then (i), (ii) and (iii) hold for $X'$ and $Y'$, and so by (ii) for the ring $R'$, they follow for $X$ and $Y$.
\end{proof}

\begin{proof}[Proof of Theorem B]
Suppose there is an exceptional $RQ$-lattice $X$ of rank vector $\alpha$. By assumption the ring $R$ is non-zero, so there exists a homomorphism $R\to K$ with $K$ an algebraically closed field. Then $X^K$ is an exceptional $KQ$-module of dimension vector $\alpha$, so $\alpha$ is a real Schur root.

Conversely, if $\alpha$ is a real Schur root, then by \cite[Theorem 1]{CBrirq}, there is an exceptional $\Z Q$-lattice $X_0$ of rank vector $\alpha$, necessarily pointwise free, and then $X = X_0\otimes_\Z R$ is a pointwise free exceptional $R Q$-lattice of rank vector $\alpha$.

Now suppose that $Y$ is a rigid $RQ$-lattice of rank vector $\alpha$. Then $P = \Hom_{RQ}(X,Y)$ 
is a projective $R$-module of constant rank $\hom(\alpha,\alpha)=1$. There is an evaluation map
\[
f : X\otimes_R P \cong X\otimes \Hom_{RQ}(X,Y) \to Y.
\]
Moreover for any homomorphism $R\to K$, we can identify $f^K$ with the evaluation map
\[
X^K \otimes_K \Hom_{KQ}(X^K,Y^K) \to Y^K.
\]
Now if $K$ is an algebraically closed field, then $X^K$ and $Y^K$ are isomorphic, so $f^K$ is an isomorphism. Thus if $R$ is reduced, we can deduce that $f$ is an isomorphism.

On the other hand, if $R$ is not reduced, we can lift $Y$ to an $R'Q$-lattice $Y'$ with $R'$ reduced, using Lemma~\ref{l:liftlatts}. Moreover $X$ lifts to the exceptional pointwise free lattice $X' = X_0\otimes_{\Z}R'$. Since $R'$ is reduced, we have $Y' \cong X'\otimes_R P'$ for some finitely generated projective $R'$-module of constant rank 1, and then $Y \cong X\otimes_R P$ where $P = (P')^R$. It follows that $Y$ is exceptional. Moreover $P$ is uniquely determined, since $P \cong \Hom_{RQ}(X,Y)$.

Finally suppose that $Y$ is rigid and pointwise free of rank vector $\alpha$. We have
\[
R^{\alpha_i} \cong e_i Y \cong e_i X\otimes_R P \cong R^{\alpha_i}\otimes_R P \cong P^{\alpha_i}
\]
for all $i$.
Since $\alpha$ is a real root for $Q$, it is indivisible, that is, its components are coprime. Thus we can find $a_i,b_i\in\N$ such that
\[
1 + \sum_{i\in Q_0} a_i\alpha_i = \sum_{i\in Q_0} b_i\alpha_i.
\]
Then
\[
P \oplus \bigoplus_{i\in Q_0} (R^{\alpha_i})^{a_i} \cong P \oplus \bigoplus_{i\in Q_0} (P^{\alpha_i})^{a_i} 
\cong \bigoplus_{i\in Q_0} (P^{\alpha_i})^{b_i} \cong \bigoplus_{i\in Q_0} (R^{\alpha_i})^{b_i}.
\]
Thus $P$ is stably free. Now any stably free projective module of constant rank 1 for a commutative ring is free, see for example \cite[Theorem 4.11]{L}, so $P\cong R$. Thus $Y\cong X$. 
\end{proof}




\begin{proof}[Proof of Theorem C]
Clearly a lattice with the indicated decomposition is rigid. For the converse, suppose that $X$ is a rigid lattice of pointwise constant rank. We fix temporarily a homomorphism $R\to K$ with $K$ an algebraically closed field. Then $X^K$ is rigid, so decomposes as a direct sum 
\[
X^K \cong M_1^{m_1} \oplus \dots \oplus M_r^{m_r}.
\]
with the $M_i$ pairwise non-isomorphic exceptional $KQ$-modules, $\Ext^1_{KQ}(M_i,M_j)=0$ for all $i,j$ and all $m_i>0$. By \cite[Corollary 4.2]{HR}, we can order the $M_i$ so that $(M_1,\dots,M_r)$ is an exceptional sequence, so $\Hom_{KQ}(M_i,M_j)=0$ for $i>j$. Let $\alpha^i = \dimv M_i$, a real Schur root, and let $X_i$ be the exceptional pointwise free $RQ$-lattice of rank vector $\alpha^i$. We have $\Hom_{RQ}(X_i,X_j)=0$ for $i>j$ since $\hom(\alpha^i,\alpha^j) = \dim \Hom_{KQ}(M_i,M_j) = 0$. Now $P_r = \Hom_{RQ}(X_r,X)$ is a finitely generated projective $R$-module of constant rank. Consider the evaluation map $\theta : X(\alpha^r) \otimes_R P_r \to X$ and let $C$ be its cokernel. For an arbitrary homomorphism $R\to K$ with $K$ an algebraically closed field (no longer the one fixed above), using Theorem~\ref{t:homprop}, we can identify $\theta^K$ with the evaluation map
\[
X(\alpha^r)^K \otimes_K \Hom_{KQ}(X_r^K,X^K) \to X^K
\]
but we know that
\[
X^K \cong (X_1^K)^{m_1} \oplus \dots \oplus (X_r^K)^{m_r}
\]
since both sides are rigid $KQ$-modules of the same dimension vector. Thus $\theta^K$ is the
inclusion of $(X_r^K)^{m_r}$ as a direct summand of $X^K$. 

Assuming that $R$ is reduced, it follows from Lemma~\ref{l:commring} that $\theta$ is a split monomorphism of $R$-modules. Thus $C$ is projective over $R$. Also
\[
C^K \cong \Coker(\theta^K) \cong (X_1^K)^{m_1} \oplus \dots \oplus (X_{r-1}^K)^{m_{r-1}}
\]
so $\Ext^1_{KQ}(C^K,C^K)=0$ and $\Ext^1_{KQ}(C^K,X_r^K)=0$. Thus $C$ is rigid and $\Ext^1_{RQ}(C,X_r)=0$. Thus $X \cong X_r \otimes_R P_r \oplus C$. Now the result follows by induction (on $r$ or on the total rank of $X$).

In case $R$ is not reduced, we have a homomorphism $R'\to R$ with $R$ reduced and $X$ lifts to a rigid $R'Q$-lattice. Moreover $X_1,\dots,X_r$ are obtained by inducing up rigid $\Z Q$-lattices, and we could equally well have induced them to $R'$ giving lifts $X_i'$ of $X_i$. Then the argument for reduced rings gives 
\[
X' \cong (X_1'\otimes_{R'} P'_1) \oplus \dots \oplus (X'_r \otimes_{R'} P'_r)
\]
and the isomorphism for $X$ follows.

Finally for uniqueness, suppose that we have two such decompositions
\[
X \cong (X_1 \otimes_R P_1) \oplus \dots \oplus (X_r\otimes_R P_r)
\cong
(Y_1 \otimes_R P'_1) \oplus \dots \oplus (Y_s\otimes_R P'_s).
\]
By considering a homomorphism $R\to K$ with $K$ algebraically closed we see that $r=s$ and we may assume that $X_i\cong Y_i$ for all $i$. Now
\[
P_r \cong \Hom_{RQ}(X_r,X) \cong P'_r
\]
and because of the decomposition of $X$, the evaluation map $X_r\otimes_R P_r \to X$ is split mono with cokernel isomorphic to
\[
(X_1 \otimes_R P_1) \oplus \dots \oplus (X_{r-1}\otimes_R P_{r-1})
\cong 
(X_1 \otimes_R P'_1) \oplus \dots \oplus (X_{r-1}\otimes_R P'_{r-1}),
\]
and then by induction $P_i\cong P_i'$ for $i=1,\dots,r-1$.
\end{proof}

For the proof of our main theorems, we did not need the fact that the braid group action works over arbitrary commutative rings, but since it may be useful for other purposes, we record it here. A sequence of exceptional $RQ$-lattices $(X_1,\dots,X_r)$ is called an \emph{exceptional sequence} provided that $\Hom_{RQ}(X_i,X_j) = \Ext^1_{RQ}(X_i,X_j)=0$ for all $i>j$.

\begin{theorem}
If $(X,Y)$ is an exceptional pair of $RQ$-lattices of pointwise constant rank, then there are exceptional pairs $(L_X Y,X)$ and $(Y,R_Y X)$ of lattices of pointwise constant rank given as follows. If $\Hom_{RQ}(X,Y)=0$ then $L_X Y$ and $R_Y X$ are given by the universal exact sequences
\[
0 \to  Y \to  L_X Y \to  X \otimes_R \Ext^1_{RQ}(X,Y) \to  0,
\]
\[
0 \to  Y \otimes_R D\Ext^1_{RQ}(X,Y) \to  R_Y X \to  X \to  0.
\]
where $D = \Hom_R(-,R)$. If $\Hom_{RQ}(X,Y)\neq 0$, then the universal map
\[
f:X \otimes_R \Hom_{RQ}(X,Y) \to  Y
\]
is an epimorphism or a monomorphism, and $L_X Y$ is its kernel or cokernel, and the universal map 
\[
g:X \to  \Hom_R(\Hom_{RQ}(X,Y),Y) \cong Y \otimes_R D\Hom_{RQ}(X,Y)
\]
is an epimorphism or a monomorphism and $R_Y X$ is its kernel or cokernel.
\end{theorem}

\begin{proof}
First a remark about the case when $R=K$, an algebraically closed field. The existence of mutations was shown in \cite[Lemma 6]{CBesrq}, but for this description of the mutations \cite{CBesrq} cites \cite{Go} and \cite{CBrirq} cites \cite{Ru}. There is now a more direct reference, namely \cite[\S4, p2290]{HubKr}.

The case when $R$ is reduced follows in the same way as \cite[Lemma 4]{CBrirq}, using Lemma~\ref{l:redring}. For $R$ not reduced, using Lemma~\ref{l:liftlatts} we lift $X$ and $Y$ to lattices $X'$ and $Y'$ for a reduced ring $R'$. If $X$ and $Y$ have rank vectors $\alpha$ and $\beta$, then so do $X'$ and $Y'$, so $\Hom_{R'Q}(Y',X')$ is a projective $R'$-module of rank $\hom(\beta,\alpha)$, the same as the rank of the $R$-module $\Hom_{RQ}(Y,X)$, which is zero. Also $X'$ and $Y'$ are exceptional, for example by Theorem~B. Thus $(X',Y')$ is an exceptional pair. Thus the mutations $L_{X'} Y'$ and $R_{Y'}X'$ are given as stated in the theorem, and the result for $L_X Y$ and $R_Y X$ follows.
\end{proof}

It follows that there is an action of the braid group on $n$ strings on the set of exceptional sequences of length $n$ consisting of pointwise constant rank $RQ$-lattices. Moreover, by the results of \cite{CBesrq}, if $Q$ is a quiver without oriented cycles and with $n$ vertices, say $1,\dots,n$, ordered so that there is no arrow from $i$ to $j$ for $j>i$, then every exceptional sequence of this type which is complete (meaning of length $n$) is in the orbit of one of the form $(Y_1, \dots, Y_n)$ where each $Y_i$ is a representation of $Q$ given by a finitely generated projective $R$-module of constant rank 1 at vertex $i$ and zero at every other vertex.


\begin{thebibliography}{99}
\frenchspacing
\raggedright

\bibitem{CK}
P. Caldero and B. Keller, 
From triangulated categories to cluster algebras. II,
Ann. Sci. École Norm. Sup. (4) 39 (2006), 983--1009.

\bibitem{CBesrq}
W. Crawley-Boevey, 
Exceptional sequences of representations of quivers, 
in `Representations of algebras' (Ottawa, ON, 1992), 117--124, CMS Conf. Proc., 14, Amer. Math. Soc., Providence, RI, 1993. 

\bibitem{CBsgrq}
W. Crawley-Boevey, 
Subrepresentations of general representations of quivers,
Bull. London Math. Soc. 28 (1996), 363--366.

\bibitem{CBrirq}
W. Crawley-Boevey, 
Rigid integral representations of quivers,
in `Representation theory of algebras' (Cocoyoc, 1994), 155--163,
CMS Conf. Proc., 18, Amer. Math. Soc., Providence, RI, 1996.

\bibitem{DG}
M. Demazure and P. Gabriel,
Groupes alg\'{e}briques. Tome I: G\'{e}om\'{e}trie alg\'{e}brique, g\'{e}n\'{e}ralit\'{e}s, groupes commutatifs,
Masson \& Cie, \'{E}diteurs, Paris; North-Holland Publishing Co., Amsterdam, 1970. xxvi+700 pp. 

\bibitem{E}
D. Eisenbud, 
Commutative algebra. With a view toward algebraic geometry, 
Graduate Texts in Mathematics, 150. Springer-Verlag, New York, 1995.
 
\bibitem{GLS}
C. Geiß, B. Leclerc and J. Schröer, 
Rigid modules and Schur roots. Math. Z. 295 (2020), no. 3-4, 1245--1277. 

\bibitem{Go}
A. L. Gorodentsev, 
Exceptional bundles on surfaces with a moving anticanonical class (Russian),
Izv. Akad. Nauk SSSR Ser. Mat. 52 (1988), 740--757, 895; 
translation in Math. USSR-Izv. 33 (1989), 67--83.
 
\bibitem{HR}
D. Happel and C. M. Ringel, 
Tilted algebras, 
Trans. Amer. Math. Soc. 274 (1982), 399--443. 

\bibitem{HubKr}
A. Hubery and H. Krause, 
A categorification of non-crossing partitions,
J. Eur. Math. Soc. 18 (2016), 2273--2313.

\bibitem{L}
T. Y. Lam, 
Serre's problem on projective modules, 
Springer Monographs in Mathematics. Springer-Verlag, Berlin, 2006.

\bibitem{Ru}
A. N. Rudakov, 
Exceptional collections, mutations and helices,
in `Helices and vector bundles', 1--6, London Math. Soc. Lecture Note Ser., 148, Cambridge Univ. Press, Cambridge, 1990.

\bibitem{S}
A. Schofield,
General representations of quivers,
Proc. London Math. Soc. (3) 65 (1992), 46--64. 

\end{thebibliography}
\end{document}